\documentclass[12pt]{amsart}
\usepackage{graphicx}
\usepackage[ansinew]{inputenc}
\usepackage{array}
\usepackage{color}
\usepackage{amsmath}
\usepackage{amsxtra}
\usepackage{amstext}
\usepackage{amssymb}
\usepackage{latexsym}
\usepackage{dsfont}
\usepackage{verbatim}
\usepackage{amsrefs}
\usepackage{mathrsfs}
\usepackage{relsize}

\allowdisplaybreaks

\newcommand*{\concat}{\mathbin{\raisebox{0.5ex}{$\smallfrown$}}}

\usepackage{fullpage}

\title{The transfinite mean}
\author{Andre Kornell}
\address{Department of Mathematics, University of California, Davis, CA 95616}
\email{kornell@math.ucdavis.edu}

 \newtheorem{thm}{Theorem}[section]
 \newtheorem{cor}[thm]{Corollary}
 \newtheorem{lem}[thm]{Lemma}
 \newtheorem{prop}[thm]{Proposition}

 \theoremstyle{definition}

 \newtheorem{defn}[thm]{Definition}

 \theoremstyle{remark}

 \numberwithin{equation}{section}

 \theoremstyle{plain}
\newtheorem*{theorem*}{Theorem}

\newcommand{\To}{\longrightarrow}
\newcommand{\R}{\mathbb{R}}

\newcommand{\N}{\mathbb{N}}
\newcommand{\card}{\mathrm{card}\,}

\newcommand{\M}{\mathbb M}
\newcommand{\Mu}{\overline{\mathbb M}}
\newcommand{\Ml}{\underline{\mathbb M}}

\newcommand{\dom}{\mathrm{dom}}
\newcommand{\<}{\langle}
\renewcommand{\>}{\rangle}
\newcommand{\restrictto}{\!\upharpoonright}
\newcommand{\ran}{\mathrm{ran}}

\begin{document}

\maketitle

\begin{abstract}
We define a generalization of the arithmetic mean to bounded transfinite sequences of real numbers. We show that every probability space admits a transfinite sequence of points such that the measure of each measurable subset is equal to the frequency with which the sequence is in that subset. We include an argument suggested by Woodin that the club filter on $\omega_1$ does not admit such a sequence of order type $\omega_1$.
\end{abstract}

\section{Introduction}
For any probability space $(\Omega, \mathcal F, \mathbb P)$ and any event $A \in \mathcal F$, the strong law of large numbers guarantees the existence of a sequence $\<x_1, x_2, \ldots\>$ from $X$ such that
\begin{equation*}\tag{$*$}\lim_{n \To \infty}\frac 1 n \sum_{i=1}^n \chi_A(x_i) = \mathbb P(A).
\end{equation*}
Here, $\chi_A(x_i) = 1$ if $x_i \in A$, and $\chi_A(x_i) = 0$ otherwise.

We ask whether it is possible to find a sequence that satisfies this property for a given \emph{family} of events. We can always find such a sequence if the family is countable, because the probability measure is countably additive. However, if the family is very large, e.g., all of $\mathcal F$, then typically there is no such sequence. For example, if $(\Omega, \mathcal F, \mathbb P)$ is the unit interval with Lebesgue measure, then equation $(*)$ fails for each of the many measurable subsets $A\subseteq [0,1]$ of positive measure disjoint from $\{x_i\,|\,i\in \N\}$.

The main result of this paper is that it is always possible to find a \emph{transfinite} sequence $\<x_\alpha\>$ so that equation $(*)$ is true for all $A \in \mathcal F$. To state this result, it is necessary to make sense of the left side of equation $(*)$ for transfinite sequences; we do so by defining the transfinite mean $\mathbb M$. We can then establish the following:
\begin{theorem*}
There is a well-ordering $\<x_\xi : \xi \in \mathfrak c\>$ of $[0,1]$ such that for all bounded Lebesgue measurable functions $f\colon [0,1] \To \R$,
$$ \int_0^1 f(t)\, dt = \M\<f(x_\xi) : \xi \in \mathfrak c\>. $$
\end{theorem*}

This paper is a rewritten version of the author's undergraduate senior thesis. I record here my heartfelt gratitude to Edward Nelson, my undergraduate advisor, whose mentorship was a boon during a turbulent period of my life. He suggested that I work with the upper mean $\Mu$, and this has significiantly simplified the presentation; the notation is also his. Also, I thank John H. Conway and W. Hugh Woodin. John Conway taught me the ordinal arithmetic that I use in this paper. Hugh Woodin suggested the argument in section 7 that not every probability space $(\Omega, \mathcal F, \mathbb P)$ can be captured by a transfinite sequence of length $\mathrm{card} \,\Omega$. Specifically, the club filter on $\omega_1$ cannot be so captured.

\subsection*{Transfinite sequences.}
We build almost entirely on the elementary theory of ordinals, as developed in any introductory set theory text such as Enderton's \emph{Elements of Set Theory} \cite{Enderton}*{chs. 7-8}. We will use lowercase Greek letters to denote ordinals.

Let $X$ be a set. A transfinite sequence in $X$ of length $\alpha$ is just a function from $\alpha$ to $X$. Thus, $X^\alpha$ is the set of all transfinite sequences in $X$ of length $\alpha$. We write $X^*$ for the class of all transfinite sequences in $X$, i.e., $X^* = \bigcup_{\alpha \in \mathrm{Ord}} X^\alpha$, and for each transfinite sequence $s$, we write $\ell(s) = \dom(s)$ for its length. For brevity, we will simply use the word ``sequence'' to refer to transfinite sequences.

The concatenation of a sequence $s_0$ of a length $\alpha_0$ and a sequence $s_1$ of length $\alpha_1$ is the sequence $s_0 \concat s_1$ of length $\alpha_0 + \alpha_1$ defined by $(s_0 \concat s_1)(\xi) = s_0(\xi)$ for $\xi < \alpha_0$ and $(s_0 \concat s_1) (\xi) = s_1(\xi-\alpha_0)$ for $\xi \geq \alpha_0$. Both ordinal addition and sequence concatenation generalize to infinite sequences of terms in the obvious way, forming unions at limit ordinals.

A nonzero ordinal is said to be \textit{indecomposable} iff it is not the sum of two strictly smaller ordinals, and similarly, we say that a nonempty sequence is \textit{indecomposable} iff it is not the concatenation of two strictly shorter sequences. Evidently, a sequence is indecomposable if and only if its length is indecomposable. By Cantor's normal form theorem, an ordinal is indecomposable if and only if it is of the form $\omega^\sigma$.

\section{The upper mean}
Every ordinal $\alpha$ is of the form $\alpha = \omega^\sigma n +
\rho$ for some indecomposable $\omega^\sigma$, natural $n$ and
remainder $\rho < \omega^\sigma$. Hence every sequence $s \in X^*$
can be uniquely decomposed as $s = s_0 \concat \cdots s_{n-1} \concat
\tilde s$, where the sequences $s_i$  have the same indecomposable
length $\omega^\sigma$, and $\tilde s$ is strictly shorter.

\begin{defn}
Let $X$ be a set and let $s \in X^*$. A finite decomposition $s =
s_0 \concat \cdots \concat  s_{n-1} \concat \tilde s$ is the
\emph{standard decomposition} of $s$ in case $s_0, \ldots, s_{n-1}$
are indecomposable sequences of equal length, and $\tilde s$ is
of strictly smaller length than $s_0$.
\end{defn}

We now define the upper mean $\Mu$, a class function from the class $\R^*_\sim$ of all bounded sequences of real numbers to $\R$.

\begin{defn}\label{defn: upper mean}
The class function $\Mu \colon \R^*_\sim \To \R$ is defined by the following transfinite
recursion scheme. Let $s = s_0 \concat \cdots \concat s_{n-1} \concat
\tilde s$ be the standard decomposition of a sequence $s\in
\R^*_\sim$. Then
\begin{enumerate}
        \item $\Mu(s) = s$ whenever $\ell(s)=1$,
        \item $\displaystyle \Mu(s) = \limsup_{\xi \To \ell(s)} \Mu(s\restrictto\xi)$
        whenever $s$ is indecomposable with $\ell(s) > 1$,
        and
        \item $\Mu(s) = \frac 1 n (\Mu(s_0) + \cdots + \Mu(s_{n-1}))$
        whenever $s$ is decomposable.
\end{enumerate}
Here, $s\restrictto \xi$ denotes the initial segment of $s$ of length $\xi$.
\end{defn}

The following two properties of $\Mu$ follow
easily by transfinite induction.

\begin{prop}\label{prop: bounded} For all $r, s \in \R^*_\sim$, we have $\displaystyle \inf_{\xi<\ell(r)}r(\xi) \leq
\Mu(r) \leq \Mu(s) \leq \sup_{\xi < \ell(s)}s(\xi)$ whenever $\ell(r) = \ell(s)$ and $r
\leq s$ pointwise.
\end{prop}

\begin{prop}\label{prop: subadditivity}
 For all $r, s \in \R^*_\sim$ of equal length, $\Mu(r +
s) \leq \Mu(r) + \Mu(s)$.
\end{prop}

\begin{thm}\label{thm: domination}
Let $r, s \in \R^*_\sim$ be such that $\ell(r) + \ell(s) = \ell(s)$.
Then $\Mu(r \concat  s) = \Mu(s)$.
\end{thm}

\begin{proof}
Proof is by transfinite induction of $\ell(s)$.

If $s$ is decomposable, let $s = s_0 \concat \cdots \concat s_{n-1}
\concat \tilde s$ be the standard decomposition of $s$.  Since the conditions $\ell(r) + \ell(s) = \ell(s)$ and $\ell(r) +
\ell(s_0) = \ell(s_0)$ are equivalent, we have by the induction
hypothesis that
 \begin{align*} \Mu( r \concat
s) & =
    \Mu( (r \concat s_0) \concat s_1 \concat \cdots \concat s_n)
    \\ & =\frac 1 n (\Mu(r \concat s_0) + \Mu(s_1) + \cdots +
    \Mu(s_{n-1}))
    \\ & = \frac 1 n (\Mu(s_0) + \Mu(s_1) +\cdots +\Mu(s_{n-1}))
    = \Mu(s)
\end{align*}
Note that $s_0$ is necessarily shorter than $s$ since $s$ is
decomposable.

If $s$ is indecomposable, then $\ell(s) = \omega^\sigma$ for some ordinal
$\sigma$. Let $r = r_0 \concat \cdots \concat r_{m-1} \concat \tilde r$
be the standard decomposition of $r$. Since $r_0$ is indecomposable,
it has length $\omega^\tau$ for some $\tau < \sigma$. If $\tau +1 <
\sigma$, then we have
\begin{align*}
\Mu(r \concat s)
    & = \limsup_{\xi \To \omega^\sigma} \Mu((r \concat s)\restrictto\xi)
    = \limsup_{\xi \To \omega^\sigma} \Mu( r \concat (s \restrictto \xi))
    = \limsup_{\xi \To \omega^\sigma} \Mu(s|_\xi)
    = \Mu (s),
\end{align*}
because for sufficiently large $\xi < \omega^\sigma$, $\ell(r) + \xi
= \xi$. Otherwise, $\omega^\sigma = \omega^{\tau +1} = \omega^\tau
\omega = \omega^\tau + \omega^\tau + \cdots$. It follows that
$s = s_0 \concat s_1 \concat \cdots$ for some $s_i$, each of
length $\omega^\tau$. Then,
\begin{align*}
\Mu(r \concat s)
    & = \limsup_{\xi \To \omega^\sigma} \Mu((r \concat s) \restrictto \xi)\
    \\ & = \limsup_{k \To \omega} \Mu(r_0 \concat \cdots \concat
    r_{m-1} \concat (\tilde r  \concat s_0) \concat s_1 \concat \cdots \concat
    s_{k-1})
    \\ & = \limsup_{k \To \omega} \frac 1 {n+k}
            (\Mu(r_0) + \cdots +\Mu(r_{n-1}) + \Mu(\tilde r \concat s_0) +
            \Mu(s_1) + \cdots + \Mu(s_{k-1}))
    \\ & = \limsup_{k \To \omega} \frac 1 {n+k}
            (\Mu(r_0) + \cdots +\Mu(r_{n-1}) + \Mu( s_0) +
            \Mu(s_1) + \cdots + \Mu(s_{k-1}))
    \\ & = \limsup_{k \To \omega} \frac 1 {k}
            ( \Mu(s_0) +
            \Mu(s_1) + \cdots + \Mu(s_{k-1}))
         =  \limsup_{\xi \To \omega^\sigma} \Mu(s \restrictto \xi)
         =  \Mu(s),
\end{align*}
where the fifth equality follows by elementary analysis since the
sequence $\Mu(s_0), \Mu(s_1), \ldots$ is bounded.
\end{proof}

\begin{cor}
Let $s$ be indecomposable with $\displaystyle \lim_{\xi \rightarrow \ell(s)} s (\xi) = c$. Then, $\Mu(s) = c$.
\end{cor}

\begin{cor} We have $\Mu(s_0 \concat \cdots \concat s_{n-1}) = \frac 1 n
(\Mu(s_0) + \cdots + \Mu(s_{n-1}))$ whenever the sequences $s_i \in
\R^*_\sim$ have equal length.
\end{cor}

\begin{cor}
For all $r, s \in \R^*_\sim$, we have $\Mu(r \concat s) = \Mu(s \concat r)$.
\end{cor}

\begin{thm}\label{thm: axioms}
The function class function $\Mu$ is uniquely defined by the following
properties:
\begin{enumerate}
        \item $\Mu(c) = c$ whenever $c \in \R = \R_\sim^1$.
        \item $\Mu(s \concat r) = \Mu(r \concat s)$
        \item $\Mu(s_0 \concat \cdots \concat s_{n-1}) = \frac 1 n (\Mu(s_0)
        + \cdots + \Mu(s_{n-1}))$ whenever the $s_i$ have equal
        length.
        \item $\displaystyle \Mu(s) = \limsup_{\xi \To \ell(s)} \Mu(s \restrictto \xi)$
        whenever $s$ is indecomposable with length greater than $1$.
\end{enumerate}
\end{thm}

\begin{proof}
We have already shown that $\Mu$ has properties (1) - (4). The uniqueness of $\Mu$ follows by transfinite induction on the length of the input sequence $s \in \R^*_\sim$. If $s$ has length $1$, then $\Mu(s)$ is uniquely determined by property (1). If $s$ has length greater than $1$, then $\Mu(s)$ is uniquely determined by the values of $\Mu$ on sequences strictly shorter than $s$: if $s$ is decomposable with standard decomposition $s = s_0 \concat \cdots \concat s_{n-1} \concat \tilde s$, then $\Mu(s) = \Mu(\tilde s \concat s_0  \concat s_1 \concat \cdots \concat s_{n-1}) = \frac 1 n (\Mu(\tilde s \concat s_0) + \Mu(s_1) + \cdots + \Mu(s_{n-1}))$ by properties (2) and (3), and if $s$ is indecomposable, then $\displaystyle \Mu(s) = \limsup_{\xi \To \ell(s)} \Mu(s \restrictto \xi)$ by property (4).
\end{proof}

\begin{defn}
Let $s \in \R^*_\sim$ be a sequence whose length is divisible by $\beta$ in the sense that for some $\alpha$, $\ell(s) = \beta \alpha$. We define $s/\beta$ to be the sequence of length $\alpha$ that results from replacing each segment of $s$ of length $\beta$ with the value of $\overline{\mathbb M}$ on that segment. In other words, if
\[ s = s_0 \concat s_1 \concat s_2 \concat \cdots \concat s_\xi \concat \cdots \quad (\xi < \alpha)\]
and each $s_\xi$ has length $\beta$, then
$$ s / \beta = \< \overline{\mathbb M} (s_0) \> \concat \< \overline{\mathbb M}(s_1)\> \concat \< \overline{\mathbb M}(s_2) \> \concat \cdots \concat \< \Mu(s_\xi) \> \concat \cdots \quad (\xi < \alpha).$$ If $s$ is bounded, then by Proposition \ref{prop: bounded} so is $s/\beta$. Thus, $s/ \beta \in \R^*_\sim$.
\end{defn}

If $\beta$ does not divide the length of $s$, then
we leave $s/\beta$ undefined. Note that more generally if bounded sequences $s_\xi$, for $\xi <\alpha$, all have lengths divisible by $\beta$, then $(s_0 \concat s_1 \concat \cdots \concat s_\xi \concat \cdots)/\beta = (s_0/\beta) \concat (s_1 /\beta) \concat \cdots \concat (s_\xi / \beta) \concat \cdots $.

\begin{thm}\label{thm: active}
For all $s \in \R^*_\sim$, if $\alpha$ and $\beta$ are ordinals such that $\beta \alpha$ divides $\ell(s)$, then $(s/\beta)/\alpha = s/(\beta\alpha)$.
\end{thm}

\begin{proof}
Note that this equality holds for all finite $\alpha$ and indecomposable $\beta$, simply by definition of $\Mu$. The proof of the general case is by transfinite induction on $\alpha$. Fix a sequence $s \in \R_\sim^*$ and ordinals $\alpha$ and $\beta$ such that $\beta\alpha$ divides $\ell(s)$. Assume that for all sequences $s' \in \R_\sim^*$ and all ordinals $\alpha'$ and $\beta'$ such that $\beta' \alpha'$ divides $\ell(s')$, if $\alpha' < \alpha$, then $(s'/\beta')/\alpha' = s'/(\beta'\alpha')$.

Without loss of generality, we may assume that $\ell(s)$ is equal to $\beta \alpha$, since this case easily implies the general case in which $\ell(s)$ is merely divisible by $\beta\alpha$. In this case, $(s/\beta)/\alpha = \< \Mu(s/\beta)\>$ and $s/\beta\alpha = \< \Mu(s)\>$. We may also assume without loss of generality that $\alpha$ is of the form $\alpha = \omega^\sigma n$, as otherwise we may establish the equality $\Mu(s/\beta) = \Mu(s)$ by neglecting remainder segments, as in Definition \ref{defn: upper mean} of $\Mu$. In particular, $\alpha$ is a limit ordinal. Hence, we may likewise assume without loss of generality that $\beta$ is of the form $\beta = \omega^\tau m$. Otherwise, it is of the form $\beta = \omega^\tau m + \rho$ for $0 < \rho < \omega^\tau$, and we may show that $s/\beta = s/\omega^\tau m$ by appealing to Lemma \ref{defn: upper mean} and to Theorem \ref{thm: domination} at each entry. This equality yields $s/\beta/\alpha = s/(\omega^\tau m)/\alpha = s/(\omega^\tau m \alpha) = s/(\beta\alpha)$.

If $\sigma = 0$, then $\alpha=n$ and $s$ is the concatenation of $mn$ bounded sequences of indecomposable length, so $s/\beta/\alpha = s/\beta\alpha$ by elementary arithmetic.

If $\sigma>0$ and $n>1$, then appealing to the fact that $\beta \omega^\sigma = \omega^{\tau + \sigma}$, we calculate that
$$
s/\beta/\alpha
    = s/\beta/\omega^\sigma n
    = s/\beta/\omega^\sigma/n
    = s/\beta\omega^\sigma/n
    = s/\beta \omega^\sigma n
    =s/\beta\alpha.
$$

If $\sigma>0$ and $n = 1$, then we first establish the equality for those $\beta$ such that $m = 1$.
\begin{align*}
s/\beta/ \alpha
 &  = s/\omega^\tau/\omega^\sigma
    = \< \Mu(s/\omega^\tau) \>
    = \< \limsup_{\xi \To \omega^\sigma} \Mu((s/\omega^\tau)\restrictto \xi) \>
    = \< \limsup_{\omega^\eta k \To \omega^\sigma}  \Mu((s/\omega^\tau)\restrictto \omega^\eta k) \>
 \\&= \< \limsup_{\omega^\eta k \To \omega^\sigma}
        (s/\omega^\tau/\omega^\eta k)(0) \>
    = \< \limsup_{\omega^\eta k \To \omega^\sigma}
        (s/\omega^{\tau +\eta} k)(0) \>
    = \< \limsup_{\omega^\eta k \To \omega^{\tau+\sigma}}
        (s/\omega^{\eta} k)(0) \>
 \\&= \< \limsup_{\omega^\eta k \To \omega^{\tau+\sigma}}
        \Mu (s \restrictto \omega^{\eta} k) \>
    = \< \limsup_{\xi \To \omega^{\tau + \sigma}} \Mu(s \restrictto \xi) \>
    = \< \Mu(s) \>
    = s/\beta\alpha.
\end{align*}
In this calculation, we have appealed to the fact that $\ell(s/\omega^\tau) = \omega^\sigma$ and therefore $\ell(s/\omega^\tau)$ is divisible by every nonzero ordinal less than $\omega^\sigma$, and similarly, to the fact that $\ell(s) = \omega^{\tau + \sigma}$ and therefore $\ell(s)$ is divisible by every nonzero ordinal less than $\omega^{\tau + \sigma}$.

If $\sigma > 0$ and $n =1$ and $m >1$, then we appeal to the previous case to calculate that 
\begin{align*}
        s/\beta/\alpha
   &=   s/\omega^\tau m /\omega^\sigma 
    =   s /\omega^\tau / m/ \omega^\sigma 
    =   s/\omega^\tau/ m/ \omega^1/ \omega^{ -1 + \sigma} 
 \\&=   s/\omega^\tau/\omega^1/\omega^{-1 + \sigma}
    =   s/\omega^{\tau + \sigma}
    =   s/\beta\alpha
\end{align*}
since $r/m /\omega = r/\omega$ for all $r \in \R_\sim^\omega$ by
elementary analysis. If $\sigma$ is finite, then $-1+\sigma$ is defined to be the predecessor of $\sigma$, and if $\sigma$ is infinite, then $-1+\sigma$ is defined to be simply equal to $\sigma$; in any case, $1+ (-1 + \sigma) = \sigma$.

Thus, we have shown that $s/\beta/\alpha = s/\beta\alpha$ for arbitrary finite $m$, establishing this equality in every remaining case.
\end{proof}

\section{Capturing spaces}

A probability space is a measure space $(X, \Sigma, m)$ such that $m(X) = 1$.

\begin{defn} Let $X$ be an arbitrary set and let $A \subseteq X$.
If $s\in X^*$, then $As \in \R^*_\sim$ denotes the sequence of the
same length as $s$, defined by $(As){(\xi)} = 1$ if $s{(\xi)} \in A$, and
$(As){(\xi)} = 0$ otherwise. In other words, $As$ is a shorthand for $\chi_A \circ s$, where $\chi_A$ is the characteristic function of $A$. 
\end{defn}

For each ordinal $\alpha$, we write $\omega_\alpha$ for the $\alpha^{\mathrm{th}}$ infinite cardinal, itself considered as an ordinal.

\begin{defn} Let $(X, \Sigma, m)$ be a probability space, and
let $A \subseteq X$ be measurable. A sequence $s \in X^*$
\emph{captures $A$ with resolution $\omega_\alpha$} in case $As/\omega_\alpha$ is a constant sequence with value $m(A)$. A sequence $s \in X^*$ captures a collection $\mathcal A$ of
measurable subsets of $X$ with resolution $\omega_\alpha$ in case it
captures each set $A \in \mathcal A$ with resolution
$\omega_\alpha$.
\end{defn}

The notion of a sequence capturing a measurable set is essentially
contained in the Strong Law of Large Numbers.

\begin{lem}\label{lem: ordering}
 Let $(X, \Sigma, m)$ be a probability space and
$\mathcal A \subseteq \Sigma$ be a countable collection.  Then there
is a sequence $s\in X^\omega$ that captures $\mathcal A$ with
resolution $\omega$.
\end{lem}

\begin{proof}
The Strong Law of Large Numbers essentially states that almost all
$s\in X^\omega$ capture any given $A \in \mathcal A$ with resolution
$\omega$. Since measure is countably additive, it follows that
almost all $s\in X^\omega$ capture all $A \in \mathcal A$ with resolution
$\omega$.
\end{proof}

We now generalize Lemma \ref{lem: ordering} to uncountable
collections $\mathcal A$.

\begin{thm}\label{thm: capture}
 Let $(X, \Sigma, m)$ be a probability space. Let
$\mathcal A_0, \mathcal A_1, \ldots ,\mathcal A_\gamma$ be a
sequence of collections $\mathcal A_\alpha \subseteq \Sigma$ with
$\card \mathcal A _\alpha \leq \omega_\alpha$. Then there exists a 
sequence $s \in X^*$ of length $\omega_\gamma$ that captures each
collection $\mathcal A_\alpha$ with resolution $\omega_\alpha$.
\end{thm}

\begin{proof}
Without loss of generality, we may assume that $\mathcal A_\alpha \subseteq \mathcal A_\beta$ whenever $\alpha \leq \beta \leq \gamma$. The proof is by transfinite induction on $\gamma$.  The $\gamma = 0$ case is Lemma \ref{lem: ordering}.

Hence, suppose that $\gamma$ is a limit ordinal. We want to write
$s$ as a concatenation $s = s_0 \concat s_1 \concat \cdots \concat s_\xi \concat \cdots \; (\xi < \gamma)$ of
sequences with $\ell(s_\xi) = \omega_\xi$, each of which will
capture an increasingly large portion of $\mathcal A_\gamma$.
Clearly there is a chain of sets $\mathcal B_0 \subseteq
\mathcal B_1 \subseteq \cdots \subseteq \mathcal B_\alpha \subseteq \cdots \; (\alpha < \gamma)$ whose union is
$\mathcal A_\gamma$ and which satisfies $\card \mathcal
B_\alpha \leq \omega_\alpha$ for $\alpha < \gamma$. Now use the induction
hypothesis to choose sequences $s_\xi$ of length $\omega_\xi$
that capture $\mathcal A_\alpha \cup \mathcal B_\alpha$ with
resolution $\omega_\alpha$ for all $\alpha \leq \xi$. The sequence $s$
captures $\mathcal A _\alpha \cup \mathcal B_\alpha$  with
resolution $\omega_\alpha$ for every $\alpha < \gamma$. To see this,
suppose that $A \in \mathcal A_\alpha \cup \mathcal B_\alpha$.
\begin{align*}
    As/\omega_\alpha & = (As_0 \concat \cdots \concat A s_\alpha \concat
            As_{\alpha +1} \concat \cdots \concat As_\xi \concat \cdots \; (\xi < \gamma))/\omega_\alpha\\
    &= (As_0 \concat \cdots \concat A s_\alpha)/\omega_\alpha \concat As_{\alpha
    +1}/\omega_\alpha \concat \cdots \concat As_\xi /\omega_\alpha \concat \cdots \; (\xi < \gamma)\\
     & =  A s_\alpha/\omega_\alpha \concat As_{\alpha
    +1}/\omega_\alpha \concat \cdots \concat As_\xi/\omega_\alpha \concat \cdots \; (\xi < \gamma)\\
    & = \< m(A) : \xi < \omega_\gamma \>
\end{align*}

Suppose now that $\gamma$ is a successor ordinal, and write $\gamma-1$ for its predecessor. We want to write $s$ as a
concatenation $s = s_0 \concat s_1 \concat \cdots \concat s_\xi \concat \cdots \; (\xi < \omega_\gamma)$, and each $s_\xi$ will have length $\omega_{\gamma-1}$, the
cardinal immediately preceding $\omega_\gamma$. We present $\mathcal
A_\gamma$ as the union of a chain $\mathcal B_0 \subseteq
\mathcal B_1  \subseteq \cdots \subseteq \mathcal B_\beta \subseteq \cdots \; (\beta < \omega_\gamma)$ such
that $\card \mathcal B_\beta \leq \omega_{\gamma-1}$. For each
$\xi < \omega_\gamma$, choose $s_\xi$ to be a sequence of
length $\omega_{\gamma-1}$ that captures $\mathcal A_\alpha$ with
resolution $\omega_\alpha$ for all $\alpha < \gamma -1$ and captures $\mathcal A_{\gamma -1 } \cup \mathcal B_\xi$ with resolution
$\omega_{\gamma -1 }$. The sequence $s$ captures $\mathcal A_\alpha$ with
resolution $\omega_\alpha$ for all $\alpha \leq {\gamma -1 }$, since each segment
$s_\xi$ does. To see this for $\alpha = \gamma $, suppose
that $A \in \mathcal B_\beta$ for some ordinal $\beta$. We now compute:
\begin{align*}
(A s)/ \omega_\gamma & = (A s_0 \concat \ldots \concat A
s_\beta\concat \cdots \concat As_\xi \concat \cdots \; (\xi < \omega_\gamma)) /\omega_\gamma
\\ & = (As_\beta \concat \cdots \concat As_\xi \concat \cdots \; (\xi < \omega_\gamma))/\omega_\gamma
\\  & = ((As_\beta \concat \cdots \concat As_\xi \concat \cdots \; (\xi < \omega_\gamma))/\omega_{\gamma -1 })/\omega_\gamma
\\ & = (As_\beta / \omega_{\gamma-1} \concat \cdots \concat As_\xi/ \omega_{\gamma-1} \concat \cdots \; (\xi < \omega_\gamma))/\omega_\gamma
\\ & = \< m(A) : \xi < \omega_\gamma \>/\omega_\gamma
= m(A).
\end{align*}
\end{proof}

\begin{defn}\label{defn: capture} A sequence $s\in X^*$ captures the measure
space $(X, \Sigma, m)$ in case it captures $\Sigma$ with resolution
$\ell(s)$, i.e., if for every set $A \in \Sigma$, we have $\Mu(As) = m(A)$.
\end{defn}

\begin{cor}\label{cor: capturing} Every probability space $(X, \Sigma, m)$ is
captured by a sequence $s \in X^*$.
\end{cor}

\begin{thm}\label{thm: ordering}
Let $(X, \Sigma, m)$ be an atomless probability space such that $\mathrm{card} \, \Sigma \leq \mathrm{card}\, X$ and furthermore $\mathrm{card} \, A = \mathrm{card} \, X$ whenever $A \in \Sigma$ has positive measure. Then, there is a well-ordering of $X$ that captures $(X, \Sigma, m)$.
\end{thm}

\begin{proof}
We adjust the proof of Theorem \ref{thm: capture} to show that if $\mathcal A_0, \mathcal A_1, \ldots, \mathcal A_\gamma$ is a sequence of collections $\mathcal A_\alpha \subseteq \Sigma$ with $\card \mathcal A_\alpha \leq \omega_\alpha$, and $Y$ is any subset of $X$ such that $\card Y < \card X$, \emph{and} $\omega_\gamma \leq \card X$, then there is a \emph{nonrepeating}, i.e., injective sequence $s \in (X \setminus Y)^*$ of length $\omega_\gamma$, which captures each collection $\mathcal A_\alpha$ with resolution $\omega_\alpha$.

The $\gamma =0$ case follows from the fact that almost all sequences in $X^\omega$ are nonrepeating when $X$ is atomless. This is easy to see when $Y$ is negligible, but $Y$ need not be negligible, as it may fail to be in $\Sigma$. In the general case, we apply our probabilistic reasoning to the probability space $X \setminus Y$. The restriction of measurable subsets to $X \setminus Y$ yields not only a $\sigma$-algebra on $X \setminus Y$, but also a probability measure on that $\sigma$-algebra. Indeed, let $A$ and $A'$ be any elements of $\Sigma$ such that $A \cap (X \setminus Y) = A' \cap (X \setminus Y)$. Their symmetric difference $A \triangle A'$ is a \emph{measurable} subset of $Y$, so $m(A\triangle A') = 0$. Thus $m(A) = m(A')$.

For the $\gamma>0$ cases, we modify our proof of Theorem \ref{thm: capture} by selecting the sequences $s_\xi$ recursively. At  the $\xi^{\mathrm{th}}$ stage of the construction, we have selected sequences $s_0, s_1, \ldots, s_\eta, \ldots \; (\eta < \xi)$, which range over a subset $Y_\xi \subseteq X$ of cardinality strictly smaller than $\omega_\gamma \leq \card X$. We then apply the induction hypothesis to the set $Y \cup Y_\xi$. By transfinite induction, we obtain the claimed variant of Theorem \ref{thm: capture}. The instance of interest is $Y = \emptyset$, $\omega_\gamma = \card \Sigma$, $\mathcal A_\gamma = \Sigma$, and $\mathcal A_\alpha = \emptyset$ for $\alpha < \gamma$. Thus, we have shown that there is a nonrepeating sequence $s \in X^*$ of length $\card X$ that captures $\Sigma$.

We now insert the elements of $X$ that do not appear in $s$,
so sparsely that they do not affect the behavior of $s$ with respect
to averaging functions. Specifically, writing $\omega_\delta$ for the cardinality of $X \setminus \ran(s)$,  we well-order the missed
elements $x_0, x_1, \ldots, x_\xi, \ldots \;(\xi< \omega_\delta)$, 
and we write $s$ as a
concatenation of sequences of length $\omega$: $s=s_1 \concat s_2
\concat \cdots \concat s_\xi \concat \cdots \; (\xi < \card X)$. Note that $\card X > \omega$ since a
countable probability space is necessarily not atomless. Now define
\begin{align*}&s' = (\<x_0\> \concat s_0 \concat \<x_1\> \concat s_1 \concat \cdots \concat \<x_\xi\> \concat s_\xi \concat \cdots \; (\xi < \omega_\delta)) \\ & \hspace{35ex}\concat (s_{\omega_\delta} \concat
s_{\omega_\delta +1} \concat \cdots \concat s_\xi \concat \cdots \; (\xi< \card X)).
\end{align*}
Each sequence of the form $\<x_\xi\> \concat s_\xi$ has length
$\omega$, so by Theorem \ref{thm: domination}, for every set $A \in \Sigma$, we have $A(\<x_\xi\> \concat
 s_\xi)/\omega = As_\xi/\omega$ for all $\xi< \omega_\delta$.  We therefore have
 that $As'/(\card X) = As'/\omega/(\card X) = As/\omega/(\card X) =
 As/(\card X) = m(A)$. We conclude that $s'$ captures $(X, \Sigma,
 m)$.
 \end{proof}
 
\begin{cor}\label{cor: completion}
There is a well-ordering of $X$ that captures the completion of $(X, \Sigma, m)$.

\end{cor}

\begin{proof}
Let $s$ be a well-ordering that captures $(X, \Sigma, M)$.
We suppose that a set $A' \subseteq X$ differs from  $A
\in \Sigma$ by a null set. Then their symmetric difference $A \triangle A'$
is a subset of some $B \subseteq X$ of measure zero. It follows
that $A \setminus B \subseteq A' \subseteq A \cup B$, which
implies that $(A \setminus B)s \leq A's \leq (A\cup B)s$. The
sets $A \setminus B$ and $A \cup B$ are measurable with measure
$m(A)$, and therefore $m(A) \leq \overline{\mathbb M}(A's) \leq m(A)$. We
conclude that $s'$ captures the completion of $(X, \Sigma, m)$.
\end{proof}

\section{The transfinite mean}

\begin{defn} For $s \in \R^*_\sim$, $\Mu(s)$ is the \emph{upper
mean} of $s$. We define $\Ml\colon \R^*_\sim \To \R$ by $\Ml(s) = -
\Mu(-s)$, where $-s \in \R^*_\sim$ is the sequence of the same length as $s$ satisfying $(-s)(\xi) = -(s(\xi))$ for all $\xi < \ell(s)$. The quantity $\Ml(s)$ is the \emph{lower mean} of $s$. If
$\Ml(s) = \Mu(s)$, then the sequence $s$ has a \emph{mean} $\M(s) =
\Mu(s)$.
\end{defn}

Observe that $\Ml(s) \leq \Mu(s)$ for all $s \in \R^*_\sim$, because $0 \leq \Mu(s) + \Mu(-s)$ by Proposition \ref{prop: subadditivity}.

\begin{prop}\label{prop: linear and uniform}
For each ordinal $\alpha$, the class function $\M$ is $\R$-linear on $\R_\sim^\alpha \cap \dom (\M)$. Furthermore, for each limit ordinal $\lambda$, if $s_0, s_1, \ldots, s_\xi, \ldots \; (\xi< \lambda)$ is a sequence in $\R_\sim^\alpha \cap \dom (\M)$ that converges uniformly to some sequence $s \in \R_\sim^\alpha$, then $\M(s)$ is defined and equal to $\lim_{\xi \to \lambda} \M(s_\xi)$.
\end{prop}

\begin{proof} This proposition follows from the subadditivity of
$\Mu$ (Proposition \ref{prop: subadditivity}).
\end{proof}

\begin{prop}\label{prop: 6.4}
If $s \in X^*$ captures the probability space $(X, \Sigma, m)$, then for all $A \in \Sigma$,
$ \M(As) = m(A)$.
\end{prop}

\begin{proof}
The family $\Sigma$ is closed under complements, so $\Mu(As) = m(A)$ and $\Mu((X \setminus A)s) = m(X \setminus A)$.
$$m(A) = 1 - m(X \setminus A) = 1 - \Mu((X \setminus A)s) = 1 + \Ml(-(X \setminus A)s) = \Ml(1 - (X \setminus A)s) = \Ml(As)$$
Therefore, $\M(As) = m(A)$.
\end{proof}

\begin{thm}\label{thm: integral}
Let $X \subseteq \R^n$ have finite positive Lebesgue measure. Then, there
exists a well-ordering $s \in X^*$ of $X$ such that \[\int_X f(t_1, \ldots, t_n) \,
dt_1 \cdots dt_n = m(X) \cdot \M(f\circ s)\] for all bounded measurable $f \colon X \To \R$.
\end{thm}

\begin{proof}
We can assume that $m(X) =1$. Let $\Sigma$ be the $\sigma$-algebra of Borel sets restricted to $X$. Then, certainly $\mathrm{card}\, \Sigma = \mathrm{card}\,X = 2^{\aleph_0}$, and $\mathrm{card} \, A = \mathrm{card} \, X = 2^{\aleph_0}$ whenever $A \in \Sigma$ has positive measure. Both facts follow from Borel determinacy, which implies that every Borel set is either countable or has the cardinality of the continuum.
Thus, $X$ has a well ordering that captures the completion of $(X, \Sigma, m)$, which includes all Lebesgue measurable subsets of $X$ (Corollary \ref{cor: completion}). The integral equation then follows from the fact that every bounded measurable function can be uniformly approximated by simple functions. 
\end{proof}

\section{The club filter}

Woodin observed that there exists a probability space $(X, \Sigma, m)$ that cannot be captured by a sequence of length $\card X$. This section records Woodin's argument as it is imperfectly recalled by the author. Jech's \textit{Set Theory} includes the elementary results about the closed unbounded filter, i.e., the club filter that we use below \cite[I.8]{Jech}.

\begin{lem}\label{lem: xi-omega}
Let $\kappa$ be a cardinal whose cofinality is greater than $\omega$, and let $s \in \R^*_\sim$ be a sequence of length $\kappa$. Then, there is an ordinal $\xi_\omega$ such that $\sup_{ \xi \geq \xi_\omega} \Mu(s \restrictto \xi) = \Mu(s)$. Furthermore, if $\Ml(s) = \Mu(s)$, then there is an ordinal $\xi_s$ such that $\M(s \restrictto \xi) = \M(s)$ for all $\xi \geq \xi_s$.
\end{lem}

\begin{proof}
By definition, $\displaystyle\Mu(s) = \limsup_{\xi \To \kappa}\Mu(s \restrictto \xi)$, so there is an increasing sequence $\xi_0, \xi_1, \ldots$ in $\kappa$ such that $\displaystyle\left|\Mu(s) - \sup_{\xi \geq \xi_n} \Mu(s \restrictto \xi)\right| < \frac 1 {n+1}$ for all $n < \omega$. We have assumed that the cofinality of $\kappa$ is greater than $\omega$, so the supremum $\xi_\omega = \bigcup_n \xi_n$ is an element of $\kappa$. Thus, for all $n  < \omega$, we have
$$ \Mu(s)= \limsup_{\xi \To \kappa}\Mu(s \restrictto \xi) \leq \sup_{\xi \geq \xi_\omega} \Mu(s \restrictto \xi) \leq \sup_{\xi \geq \xi_n} \Mu(s \restrictto \xi) \leq \Mu(s) + \frac 1 {n+1},$$
so $\Mu(s) = \sup_{ \xi \geq \xi_\omega} \Mu(s \restrictto \xi)$. Replacing $s$ with $-s$, we find that there is an ordinal $\xi_\omega^-$ such that $\Ml(s) = \inf_{\xi \geq \xi_\omega^-} \Ml(s \restrictto \xi)$.

Assume that $\Ml(s) = \Mu(s)$, and let $\xi_s$ be the larger of $\xi_\omega$ and $\xi_\omega^-$. We calculate that
$$\sup_{\xi \geq \xi_s} \Mu(s \restrictto \xi) \leq \sup_{\xi \geq \xi_\omega} \Mu(s \restrictto \xi) = \Mu(s) = \Ml(s) = \inf_{\xi \geq \xi_\omega^-} \Ml(s \restrictto  \xi) \leq \inf_{\xi \geq \xi_s} \Ml(s \restrictto \xi) \leq \inf_{\xi \geq \xi_s} \Mu(s \restrictto \xi).$$
\end{proof}

\begin{prop}\label{prop: club}
Let $\kappa$ be a cardinal whose cofinality is greater than $\omega$, and let $s \in \R^*_\sim$ be a sequence of length $\kappa$. Then, the set $\{\xi \in \kappa\, |\, \Mu(s \restrictto \xi) = \Mu(s)\}$ has a subset that is club in $\kappa$.
\end{prop}

\begin{proof}
Applying Lemma \ref{lem: xi-omega}, let $\xi_\omega \in \kappa$ be an ordinal such that $\sup_{ \xi \geq \xi_\omega} \Mu(s \restrictto \xi) = \Mu(s)$. 
Note that in fact $\Mu(s) = \sup_{ \xi \geq \eta} \Mu(s \restrictto \xi)$ for any $\eta \geq \xi_\omega$, since $\Mu(s) = \limsup_{\xi \To \kappa} \Mu(s \restrictto \xi)$. 
Let $C$ be the set of indecomposable ordinals $\xi$ between $\xi_\omega$ and $\kappa$ that satisfy $\Mu(s \restrictto \xi) = \Mu(s)$.

To show that $C$ is closed, let $B \subseteq C$ be a bounded nonempty subset. If $B$ does not contain its own supremum, then there is an increasing sequence of indecomposable ordinals $\omega^{\sigma_0}, \omega^{\sigma_1}, \ldots, \omega^{\sigma_\eta}, \ldots \; (\eta < \lambda)$ in $B$ that converges to $\sup B$. Ordinal exponentiation is a normal operation, so $\sup B = \omega^{\lim_\eta \sigma_\eta}$ is indecomposable, and $\Mu(s \restrictto {\sup B}) = \limsup_{\xi \To \sup B} \Mu(s \restrictto \xi)$. Thus, by definition of $C$, the quantity $\Mu(s \restrictto {\sup B})$ is at least $\Mu(s)$, and since $\sup B \geq \xi_\omega$, the quantity $\Mu(s \restrictto {\sup B})$ is at most $\Mu(s)$, so $\Mu(s \restrictto {\sup B}) = \Mu(s)$. We conclude that $C$ is closed.

To show that $C$ is unbounded, let $\eta_0$ be any element of $\kappa$ larger than $\xi_\omega$. We may choose a strictly increasing sequence $\eta_0, \eta_1, \ldots$ such that $\Mu(s \restrictto {\eta_0}), \Mu(s \restrictto {\eta_1}), \ldots$ is a monotonically increasing sequence converging to $\Mu(s)$. We may choose each ordinal $\eta_k$ to be of the form $\omega^{\sigma_k}n_k$, by discarding remainders. If $\sigma_0, \sigma_1, \ldots$ is eventually constant with value $\sigma$, then $\lim_k \eta_k = \omega^{\sigma+1}$. Otherwise $\lim_k \eta_k = \omega^{\lim_k\sigma_k}$. In either case, we conclude that $\eta_\omega = \lim_k \eta_k$ is indecomposable. It is also an element of $\kappa$, because the latter has cofinality greater than $\omega$. We calculate that
$$ \Mu(s) \geq \Mu(s \restrictto {\eta_\omega}) = \limsup_{\xi \To \eta_\omega} \Mu(s \restrictto \xi) \geq \lim_{k \To \omega} \Mu(s \restrictto {\eta_k}) = \Mu(s).$$ 
The first inequality is true because $\eta_\omega> \xi_\omega$, and $ \sup_{\xi \geq \xi_\omega} \Mu(s \restrictto \xi) = \Mu(s)$, by choice of $\xi_\omega$. We conclude that $\eta_\omega> \eta_0$ is an element of $C$. Therefore, $C$ is unbounded.
\end{proof}

\begin{thm}
Let $\kappa$ be an uncountable cardinal whose cofinality is $\kappa$. Let $$\mathrm{club}(\kappa) = \{ A \subseteq \kappa \,|\, C\subseteq A \text{ for some club set C}\}.$$
Let $\Sigma$ be the $\sigma$-algebra generated by $\mathrm{club}(\kappa)$. Let $m : \Sigma \To \{0,1\}$ be defined by $m(A) = 1$ if and only if $A \in \mathrm{club}(\kappa)$. Then, $(\kappa, \Sigma, m)$ is a measure space that is not captured by any sequence of length $\kappa$.
\end{thm}

\begin{proof}
Write $\mathrm{club}^c(\kappa)$ for the set of complements of sets in $\mathrm{club}(\kappa)$. The intersection of countably many club sets is club, so $\Sigma = \mathrm{club}(\kappa) \cup \mathrm{club}^c(\kappa)$. Indeed, if a countable family of sets in $\Sigma$ includes an element of $\mathrm{club}^c(\kappa)$, then its intersection is in $\mathrm{club}^c(\kappa)$, and otherwise, the family consists of elements of $\mathrm{club}(\kappa)$ so its intersection is in $\mathrm{club}(\kappa)$.

Suppose that there is a sequence $s \in {\kappa}^{\kappa}$ that captures $(\kappa, \Sigma, m)$. For each $A \in \mathrm{club}(\kappa)$, we have $\M(As) = 1$, as in Proposition \ref{prop: 6.4}. Invoking Lemma \ref{lem: xi-omega}, let $\phi(A) \in \kappa$ be the least ordinal such that $\M(As \restrictto \xi) = 1$ for all $\xi \geq \phi(A)$. The supremum of the family $\{\phi(A) \, | \, A \in \mathrm{club}(\kappa)\}$ is certainly $\kappa$, since the measure space $(\kappa, \Sigma, m)$ cannot be captured by a sequence of smaller length; indeed, all subsets of $\kappa$ of smaller cardinality are measure zero. Thus, this family contains ordinals arbitrarily large in $\kappa$. For each, $\alpha \in \kappa$, choose a set $A_\alpha \in \Sigma$ such that $\phi(A_\alpha ) > \omega^{\alpha +1} $. Note that $\alpha + \omega^{\alpha +1} = \omega^{\alpha +1}$, because $\omega^{\alpha+1}$ is an indecomposable ordinal larger than $\alpha$.

The diagonal intersection $D = \Delta_{\alpha \in \kappa} A_\alpha$ is also in $\mathrm{club}(\kappa)$. By definition of diagonal intersection, $D \subseteq \{1, \ldots, \phi(D)\} \cup A_{\phi(D)}$, so for all $\xi \in \kappa$,
$$\M(Ds \restrictto \xi) \leq \M((\{1, \ldots, \phi(D)\} \cup A_{\phi(D)})s \restrictto \xi).$$
Since $\phi(D) + \omega^{\phi(D) +1} = \omega^{\phi(D) +1}$, we may neglect initial sequences of length $\phi(D)$ whenever $\xi \geq \omega^{\phi(D)+1}$; for such $\xi$, we have $1 = \M(Ds \restrictto \xi) \leq \M(A_{\phi(D)}s \restrictto \xi)$. Thus, $\M(A_{\phi(D)}s \restrictto \xi) = 1$ for all $\xi \geq \omega^{\phi(D)+1}$, which implies that $\phi(A_{\phi(D)}) \leq \omega ^{\phi(D)+1}$. This contradicts our choice of $(A_\alpha : \alpha \in \kappa)$.
\end{proof}

\end{document}